\documentclass[12pt]{article}

\usepackage{setspace}
\usepackage{a4wide}

\usepackage{amsthm, amssymb, amstext}
\usepackage[fleqn]{amsmath}
\usepackage{latexsym}
\usepackage[dvips]{graphicx}
\usepackage{comment}
\usepackage{hyperref}
\usepackage{mathtools}
\usepackage{enumerate} 

\usepackage{todonotes}
\usepackage{comment}

\newtheorem{theorem}{Theorem}
\newtheorem{lemma}{Lemma}

\newtheorem{problem}{Problem}

\theoremstyle{definition}

\usepackage{amsmath}
\usepackage{graphicx}

\makeatletter
\newcommand{\ostar}{\mathbin{\mathpalette\make@circled\star}}
\newcommand{\make@circled}[2]{%
  \ooalign{$\m@th#1\smallbigcirc{#1}$\cr\hidewidth$\m@th#1#2$\hidewidth\cr}%
}
\newcommand{\smallbigcirc}[1]{%
  \vcenter{\hbox{\scalebox{0.77778}{$\m@th#1\bigcirc$}}}%
}
\makeatother

\title{Decomposing a triangle-free planar graph into a forest and a subcubic forest}
\author{Carl Feghali\thanks{Computer Science Institute of Charles University, Prague, Czech Republic, and Univ Lyon, EnsL, CNRS, LIP, F-69342, Lyon Cedex 07, France email: \texttt{feghali.carl@gmail.com} }\and
Robert \v{S}\'amal\thanks{Computer Science Institute of Charles University, Prague, Czech Republic, email: \texttt{samal@iuuk.mff.cuni.cz} }}

\date{}

\begin{document}
\maketitle

\begin{abstract}
We strengthen a result of Dross, Montassier and Pinlou (2017) that the vertex set of every triangle-free planar graph can be decomposed into a set that induces a forest  and a set that induces a forest with maximum degree at most $5$, showing  that $5$ can be replaced by $3$. 
 \end{abstract}


The subject of this paper is the following problem, posed by Dross, Montassier and Pinlou \cite{dross}. 

\begin{problem}\label{problem}
Find the smallest integer $d \geq 0$ such that that the vertex set of every triangle-free planar graph can be decomposed into a set that induces a forest with maximum degree at most $d$ and a set that induces a forest. 
\end{problem}

In their same paper, they showed, via the discharging method, that $d \leq 5$. It is worth mentioning that their result is tight in the sense that there are triangle-free planar graphs that cannot be decomposed into two subgraphs, each having bounded maximum degree~\cite{montassier}. 

The analogue of Problem \ref{problem} for planar graphs with girth at least five was settled by Borodin and Glebov \cite{borodin1}, showing that $d=0$. Later on, Kawarabayashi and Thomassen \cite{thomassen2} gave another proof of this result using Thomassen's list color method \cite{thomassen1}. In that same paper, Kawarabayashi and Thomassen had already asked whether $d = 0$ in Problem~\ref{problem} is possible; if true, it would be a strengthening and a new proof  of a fundamental theorem of Gr\"otzsch \cite{gr} that every triangle-free planar graph admits a proper $3$-coloring. 

 In this paper, we use a list color method similar to Kawarabayashi and Thomassen \cite{thomassen2}  to address Problem \ref{problem}, by  improving the aforementioned best known upper bound of $5$ on $d$ to $3$.

\begin{theorem}\label{thm:main}
Every triangle-free planar graph can be decomposed into a forest and a forest with maximum degree at most three. 
\end{theorem}

We finally mention the related problem that asks to find the smallest integers $d, k \geq 0$ such that the vertex set of every planar graph can be decomposed into a set that induces a $d$-degenerate graph and a  set that induces a $k$-degenerate graph. This was settled by Thomassen \cite{thomassen3, thomassen} and leads to new proofs of the $5$-color theorem for planar graphs. Towards a generalization, an outstanding question due to Borodin \cite{borodin2} from 1976 asks whether every planar graph has a $5$-coloring such that for $1 \leq k \leq 5$ any $k$ color classes induce a $(k-1)$-degenerate graph.  To keep this paper short, we refer the reader to \cite{cranston, dross, thomassen2} for further details surrounding Problem \ref{problem},  its variants, generalizations and implications. 

\section{The proof of Theorem \ref{thm:main}}

In this section, we prove Theorem \ref{thm:main}.  We start with some definitions. 

A \emph{scene} $S=(G,P, Q, \delta)$ consists of
\begin{itemize}
\item a connected triangle-free plane graph $G$ and a (possibly empty) subset $P \subseteq V(G)$, with $G[P]$ consisting of 
  vertices consecutive in $V(K)$, the boundary of the outer face~$K$ of $G$,
\item a (possibly empty) subset $Q \subseteq V(K) \setminus P$ that induces an independent set {\bf\boldmath in $G$}, and
\item a (not necessarily proper) coloring $\delta: P \rightarrow \{1, 2\}$. 
\end{itemize}

A (not necessarily proper) coloring $\varphi: V(G) \rightarrow \{1, 2\}$ is \emph{valid} for a scene $S = (G, P, Q, \delta)$ if $\varphi \restriction P = \delta$, $\varphi(Q) = \{2\}$ and
\begin{itemize}
\item[(G1)] $\varphi^{-1}(1)$ induces a forest with maximum degree at most three and $\varphi^{-1}(2)$ induces a forest;
\item[(G2)] if $\delta(v) = 1$ for some $v \in P$ and $uv \in E(G)$ for some $u \in V(G) \setminus P$, then $\varphi(u) \not=1$;
\item[(G3)] if $|P| \leq 2$, $\delta(v) = 1$ for \emph{exactly} one vertex $v \in P$ and $uv \in E(G)$ for some $u \in Q$, 
  then for every $w \in N_G(v) \cap V(K) \setminus \{u\}$ there is no $\varphi$-monochromatic path with ends $u$ and~$w$;
\item[(G4)] if $\varphi(v) = 1$ for some $v \in V(K) \setminus P$, then $v$ has at most two neighbors in $\varphi^{-1}(1)$. 
\end{itemize} 

The following lemma will be used repeatedly. 

\begin{lemma}\label{obs}
  Let $S = (G, P, Q, \delta)$ be a scene. Suppose $G = G_1 \cup G_2$ for induced subgraphs $G_1$ and $G_2$ of $G$ with $E(G) = E(G_1) \cup E(G_2)$. 
  Suppose further that $P \subseteq V(G_1)$ and $H = G_1 \cap G_2$ is a path of length at most $1$ or of length $2$ and with one end in $Q$ or a cycle.  
  Suppose also that $H$ is contained in the boundary of the outer face of~$G_2$ and that 
$S_1 = (G_1, P, Q \cap V(G_1), \delta)$ has a valid coloring $\varphi_1$ and
  \[
S_2=\left\{\begin{array}{ll}(G_2, V(H) \setminus Q, Q \cap  V(G_2), \varphi_1 \restriction (H - Q)), & |V(H)| \leq 2\\ (G_2, V(H) \setminus \{s\}, (Q \cap  V(G_2))\setminus \{t\}, \varphi_1 \restriction (H - s)),& |V(H)| = 3 \mbox{ with } H = sxt \mbox{ and } s \in Q\\
(G_2, V(H), (Q \cap  V(G_2)) \setminus V(H), \varphi_1 \restriction H)
, & \mbox{ $H$ is a cycle}\end{array}\right.
\]
has a valid coloring~$\varphi_2$.

Put $\varphi = \varphi_1 \cup \varphi_2$. Then $\varphi$ with~$S$ satisfy (G2) and~(G4). Moreover, if in addition,
\begin{itemize}
  \item[\normalfont{($\star$)}] $H$ is a path and $H - q$ is $\varphi_1$-monochromatic (or possibly empty) for some vertex $q$ that is an end-vertex of $H$, 
\end{itemize}
then $\varphi$ also satisfies (G1) and~(G3) and is thus valid for~$S$. 
\end{lemma}

\begin{proof}
  In the proof we will use $K$ for the outer face of~$G$, and $K_i$ for the outer face of~$G_i$ (for $i=1, 2$). 
  To show that $\varphi$ satisfies (G1) if ($\star$) holds, we first show that $\varphi^{-1}(2) = \varphi_1^{-1}(2) \cup \varphi_2^{-1}(2)$ induces a forest. This follows as $\varphi_i^{-1}(2)$ induces a
  forest for $i = 1, 2$; moreover, by ($\star$) these two forests intersect in  a path. For $\varphi^{-1}(1)$ we argue similarly, but we also show, independently of ($\star$), that the degrees of vertices
  in $\varphi_1^{-1}(1) \cap \varphi_2^{-1}(1)$ are upper bounded by~$3$. (For other vertices of~$\varphi_1^{-1}(1) \cup \varphi_2^{-1}(1)$ it follows directly by assumption.) 
  For $i \in \{1, 2\}$ let $F_i$ denote the forest induced by $\varphi_i^{-1}(1)$, and let us consider an arbitrary vertex 
  $y \in V(F_1) \cap V(F_2)$. Then by (G1) applied to $S_1$, $\mbox{deg}_{F_1}(y) \leq 3$, and by (G2) applied to $S_2$, 
  together with the fact that $G_1$ is an induced graph, there is no neighbor of~$y$ in $V(F_2) \setminus V(F_1)$. 

  To show (G4) consider a vertex $v \in V(K)$. Observe that either $v \in V(K_1)$ or $v \in V(K_2) \setminus V(H)$ or $v \in V(H)$. 
  The first two cases follow by assumption, so consider $v \in V(H)$. 
  Recall the definition of $F_1$, $F_2$ from the previous paragraph. 
  We are bounding the degree of a vertex $v \in V(F_1) \cap V(F_2) \cap V(K)$. Then by (G4) applied to $S_1$, $\mbox{deg}_{F_1}(v) \leq 2$, and by (G2) applied to $S_2$,
  together with the fact that $G_1$ is an induced graph, there is no neighbor of~$v$ in $V(F_2) \setminus V(F_1)$, as needed.  

  We now show that $\varphi$ satisfies (G2). Consider a vertex $v \in P$ with $\delta(v) = 1$. Observe by (G2) applied to $S_1$ that $\varphi(u) = 2$ 
  for each $u \in N_{G_1}(v) \setminus V(P)$. Moreover, if $v \in V(H)$, then by (G2) applied to $S_2$ we also have $\varphi(w) = 2$ for each $w \in N_{G_2} \setminus V(H)$. 
  These two facts combined imply $\varphi$ satisfies (G2), as required. 

  We finally show that $\varphi$ satisfies (G3) if $(\star)$ holds. Let $v$ be the unique vertex of $P$ with $\delta(v) = 1$, and suppose by contradiction that there is a $\varphi$-monochromatic path $A$
  with ends $u \ne w$ for some $u \in N_G(v) \cap Q$ and $w \in N_G(v) \cap V(K)$.  If $V(A) \subseteq V(G_1)$ then
  we reach a contradiction with (G3) applied to $S_1$. Similarly, if $V(A) \subseteq V(G_2)$ then we reach a contradiction with (G3) applied to $S_2$.    So we can assume
  that the path $A$ intersects both $G_1$ and $G_2 - G_1$ and hence also intersects the path $H$ in at least one vertex. We let $h$ denote the first vertex of $H$ in common with $A$ if one traverses $A$
  from $u$ towards $w$ (possibly $h \in \{u, w\}$)  There are several possibilities to consider. 

  Suppose first that $v \in V(G_1) \setminus  V(H)$. Then $u$ and $w$ are both vertices of $G_1$. In this case, since $A$ intersects $G_1$ and $G_2- G_1$, there must be at least one vertex of $H - h$ that also belongs to $A$. We let $h'$ be the last vertex of $H - h$ in common with $A$ if one traverses $A$ from $u$ towards $w$ (possibly $h' = w$), and we let $B$ be the subpath of $H$ with ends $h$ to $h'$. By ($\star$), $B$ is $\varphi_1$-monochromatic with color $2$; thus we can find a $\varphi_1$-monochromatic path in $G_1$ by color $2$ with ends $u$ and $w$  and vertex set $V(G_1) \cap (V(A) \cup V(B))$, which contradicts (G3) in $S_1$.  The same argument can also be applied whenever
  \begin{itemize}
  \item $v \in V(H)$ and $u$ and $w$ are both vertices of $G_1$;
  \item $v \in V(H)$ and $u$ and $w$ are both vertices of $G_2$. 
  \end{itemize}

  The final case $v \in V(H)$, $u \in V(G_i - G_{3 - i})$ and $w \in V(G_{3 - i} - G_i)$ for some  $i \in \{1, 2\}$ is similar but requires a little more care.  
  Observe that $H$ cannot be a single vertex. Let $z$ denote the neighbor of $v$ on $H$, and note that $z \notin \{u, w\}$. 
  By (G2) applied to $S_1$ (and by assumption of (G3) in case $P \subseteq H$), we get $\varphi_1(z) = 2$.
  It follows that in $(\star)$ we must have $q = v$, so $H - v$ is $\varphi_1$-monochromatic with color $2$. We define
  $D$ to be the subpath of $H$ with ends $h$ and $z$ if $h \not=z$ and by $D = z$ otherwise. Observe then that there is a  $\varphi_i$-monochromatic path in $G_i$ with color $2$ with ends $u$ and $z$ and
  vertex set $V(G_i) \cap (V(A) \cup V(D))$, which contradicts (G3) in $S_i$. This completes the proof of the lemma. 
\end{proof}

A scene $S = (G, P, Q, \delta)$ is \emph{good} if 
\begin{itemize}
  \item $|P| \leq 2$ or,
  \item $|P| = 3$ and $P$ can be labeled $xyz$ so that $\delta(x) = 1$ and $\delta(y) = \delta(z) = 2$. 
\end{itemize}

Theorem \ref{thm:main} will clearly follow (by considering each connected component separately) from the following more general result. 

\begin{theorem}\label{thm:decomp}
  Every good scene $S = (G, P, Q, \delta)$ has a valid coloring.  
\end{theorem}
 
In the rest of this paper, we prove Theorem \ref{thm:decomp}. The proof is by contradiction. 
 
A \emph{counterexample} is a good scene $S = (G, P, Q, \delta)$ without a valid coloring.  We say $S$~is a \emph{minimal counterexample}
if $S$ is a counterexample such that $|V(G)|$ is minimum among all counterexamples. 

We now show, in a series of lemmas, that a  minimal counterexample has a particular structure. 

\begin{lemma}\label{lem:2connected}
  Suppose $S=(G,P, Q, \delta)$ is a minimal counterexample.
  Then $G$ is $2$-connected, the cycle $K$ bounding the outer face of $G$ is induced if $|P| \leq 2$, and every chord of $K$ has one end being the middle vertex of $P$ if $|P| = 3$. 
\end{lemma}

\begin{proof}
  Note that $G$ is connected by the definition of a scene. 

  Let $H$ consist of either a single vertex or an edge of $G$ with no end being the middle vertex of $P$ if $|P| = 3$, and $G_1$ and $G_2$ be (induced) subgraphs of $G$ such that $G = G_1 \cup G_2$ and $V(G_1) \cap V(G_2) = V(H)$.  Let us first consider the case $|P| = 3$ and $H$ is a single vertex being the middle vertex of $P$. For $i \in \{1, 2\}$, let $P_i = V(G_i) \cap P$; the scene $S_i = (G_i, P_i, Q \cap V(G_i), \delta \restriction P_i)$ has, by the minimality of $S$, a valid coloring $\varphi_i$. Clearly, $\varphi = \varphi_1 \cup \varphi_2$ is valid for $S$, a contradiction. 

  In all other cases, we may assume $P \subseteq V(G_1)$.  Let $S_1 = (G_1, P, Q \cap V(G_1), \delta)$.  By the minimality of $S$, $S_1$ has a valid coloring $\varphi_1$.  
  Let $S_2 = (G_2, V(H), Q \cap V(G_2) \setminus V(H), \varphi_1 \restriction H)$. By the minimality of $S$ again, $S_2$ has a valid coloring $\varphi_2$. 
  By Lemma \ref{obs},  the coloring $\varphi = \varphi_1 \cup \varphi_2$ is valid for $S$, which is a contradiction. This completes the proof.
\end{proof}

A vertex of a plane graph $G$ is \emph{internal} if it is not incident with the outer face of $G$. 

\begin{lemma}\label{lem:4cycle}
Suppose $S=(G,P, Q, \delta)$ is a minimal counterexample.
Then every $4$-cycle in $G$ bounds a face. 
\end{lemma}

\begin{proof}
Suppose for a contradiction that there is a $4$-cycle $C$ that does not bound a face, $G_2$ is the subgraph of $G$ drawn in the closed disk bounded by $C$, and $G_1 = G - (V (G_2) \setminus V (C))$. 
 Note that $P \subseteq V(G_1)$.  Let $S_1$ be the scene $(G_1, P, Q \cap V(G_1), \delta)$.  By the minimality of $S$, $S_1$ has a valid coloring $\varphi_1$. Let $x$, $s$, $y$, $t$ denote the vertices of $C$ in clockwise order. Note by (G1) that $\varphi_1(V(C)) = \{1, 2\}$.  Assume without loss of generality that $\varphi_1(y) = 1$. We consider all possible colorings (up to symmetry) of $V(C)$ by $\varphi_1$. 

 Suppose first $\varphi_1(s) = \varphi_1(t) = 2$. Let $G_2' = G_2 - \{y\}$, let $Y$ be the set of internal vertices of $G_2$ that are neighbors of $y$, and let $Q' = \{s, t\} \cup Y$. Since $G$ is
  triangle-free, $Q'$ is an independent set. Let $S'_2$ be the scene $(G'_2, \{x\}, Q', \varphi_1 \restriction \{x\})$, and note by the minimality of $S$ that  $S'_2$ has a valid coloring $\psi_2$. We
  define a coloring $\varphi_2$ of $S_2$ by $\varphi_2 \restriction G'_2 = \psi_2$ and $\varphi_2(y) = 1$. 
  Clearly, $\varphi_2$ is valid for $S_2$ (as defined in Lemma~\ref{obs}): the graph induced by~$\varphi_2^{-1}(1)$ differs from 
  the one induced by~$\psi^{-1}(1)$ by adding an isolated vertex. 
 Let  $\varphi$  be a coloring for $S$ defined by $\varphi = \varphi_1 \cup \varphi_2$. We claim that $\varphi$ is valid, which is a contradiction.  By Lemma \ref{obs}, $\varphi$ satisfies (G2) and (G4). 
 
To show that $\varphi$ satisfies (G1), observe that $\varphi^{-1}_2(1)$ induces a forest with maximum degree at most  three where $x$ is an isolated vertex by (G2) and $y$ is an isolated vertex by definition.  Therefore, $\varphi^{-1}(1)$ induces a forest with maximum degree at most three. 

In the other case, the sets $\varphi_1^{-1}(2)$ and $\varphi_2^{-1}(2)$ each induce  a forest and have two vertices, $s$ and $t$, in common. If $\varphi_1(x) = 2$, then 
  $V(G_1) \cap V(G_2) \cap \varphi^{-1}(2) = \{x, s, t\}$ implies $\varphi^{-1}(2)$ induces a forest. 
  On the other hand, if $\varphi_1(x) = 1$, (G3) with $v =x$, $u =s$ and $w = t$ applied to $S'_2$ ensures that $s$ and $t$ are in
  distinct components of the forest induced by $\varphi_2^{-1}(2)$. Thus $\varphi_1^{-1}(2) \cup \varphi_2^{-1}(2)$ again  induces a forest, establishing (G1). 
 
Suppose for a contradiction that (G3) does not hold in $\varphi$, that is, there is a $\varphi$-monochromatic path $A$ by color 2 in $G$ as in the statement of (G3). Since $\varphi_1$ is valid,   $A$ is not a path of $G_1$. Thus,  for $A$ to be a path in $G$, the only possibility is if $\varphi_1(x) = 1$ and  there is a $\varphi_2$-monochromatic subpath of $A$ in $G_2$ with ends $s$ and $t$, which is impossible by the previous paragraph. This establishes (G3).

By symmetry it remains to consider the following cases:
\begin{itemize}
\item  $\varphi_1(s) = \varphi_1(x) = 1 \not= \varphi_1(t) = 2$;
\item $\varphi_1(s) = 1 \not= \varphi_1(x) = \varphi_1(t) = 2$. 
 \end{itemize} 
 
We address both cases simultaneously. 
Let $G_2' = G_2 - \{y\}$, let $Y$ be the set of internal vertices of $G_2$ that are neighbors of $y$, and let $Q' = \{t\} \cup Y$. Since $G$ is triangle-free, $Q'$ is an independent set. 
Let $S'_2 = (G'_2, \{x, s\}, Q', \varphi_1 \restriction \{x, s\})$. By the minimality of $S$,  $S'_2$ has a valid coloring $\psi_2$. 
We define a coloring $\varphi_2$ of $S_2$ by $\varphi_2 \restriction G'_2 = \psi_2$ and $\varphi_2(y) = \varphi_1(y)$.  
Let $\varphi$ be a coloring for $S$ defined by $\varphi = \varphi_1 \cup \varphi_2$.  
  By the construction and by~(G2) applied to $S'_2$ we get the following that will be used repeatedly: 
\begin{itemize}
  \item[\normalfont{(A)}] Vertices in $V(H) = \{x,s,y,t\}$ of color~1 have no neighbors of color~1 in~$V(G_2) \setminus V(G_1)$. 
\end{itemize}
  Another easy observation that holds just by assumptions on colors of $V(H)$: 
\begin{itemize}
  \item[\normalfont{(B)}] Any path in $G_2 \cap \varphi^{-1}(2)$ can be rerouted along the boundary of~$K_2$, that is along $V(H)$
    (and still it will be a monochromatic path of color~2). The same is true for color~1.
\end{itemize}
Now to verify that the constructed coloring~$\varphi$ is valid for~$S$: 
  For (G1) we use claim~(B) above to show that both color sets induce forests. To show that $\varphi^{-1}(1)$ has degrees bounded by~3 we use claim~(A). 
  For (G2): if $v \notin V(H)$ we only use that $\varphi_1$ is valid for~$S_1$. 
  If $v \in V(H)$ we also need to use claim~(A). 
  To show (G3) it is enough to use claim~(B) in the same way as above. 
  Finally, (G4) follows directly from claim~(A). 
This completes the proof. 
\end{proof}

\begin{lemma}\label{lem:extended}
  Suppose $S=(G,P, Q, \delta)$ is a minimal counterexample, $K$ is the cycle bounding the outer face of $G$, and $x$ is an internal vertex of $G$ having two non-adjacent neighbors $s$ and $t$ in $K$
  distinct from the middle vertex of $P$ if $|P| = 3$. Then $\{s, t\} \cap Q = \emptyset$. 
\end{lemma}

\begin{proof}
  Suppose for a contradiction that this is not the case. The closure of the complement of~$K$ contains all of the graph~$G$. It is divided by the path~$sxt$ to two regions. 
  We let the graphs induced be these two regions be~$G_1$ and~$G_2$, with the path $sxt$ contained in both of them. We choose the labeling so that $V(P) \subseteq V(G_1)$. 
  As $s$, $t$ are nonadjacent, both $G_1$ and~$G_2$ are induced subgraphs. 
  Suppose for a contradiction (by symmetry) that $s \in Q$. Let $S_1$ be the scene $(G_1, P, Q \cap V(G_1), \delta)$.  
  By the minimality of $S$, $S_1$ has a valid coloring $\varphi_1$. Let $S_2 = (G_2, \{x, t\}, Q \cap (V(G_2) - \{t\}), \varphi_1 \restriction \{x, t\})$. By the
  minimality of $S$, $S_2$ has a valid coloring $\varphi_2$. We claim that $\varphi = \varphi_1 \cup \varphi_2$ is valid for $S$, which is a contradiction. By Lemma \ref{obs}, $\varphi$ satisfies (G2) and
  (G4). Moreover, by the same lemma  with $q = s$, $\varphi$ is valid if $\varphi_1(x) = 2$ or if $\varphi_1(x) = 1$ and $\varphi_1(t) = 1$.

  We finally show that (G1) and (G3) also hold if $\varphi_1(x) = 1$ and $\varphi_1(t) = 2$. To establish (G1), note by (G4) applied to $S_1$ that $\varphi^{-1}_1(1)$ induces a forest $F_1$ with maximum degree at most three where $d_{F_1}(x) \leq 2$ and by (G2) applied to $S_2$ that  $\varphi^{-1}_2(1)$ induces a forest $F_2$ with maximum degree at most three where  $d_{F_2 - F_1}(x) = 0$. Thus, $\varphi^{-1}(1)$ induces a forest with maximum degree at most three.  
 
  Finally,  $\varphi_1^{-1}(2)$ and $\varphi_2^{-1}(2)$ each induce  a forest; moreover  
  (G3) for $S_2$ with $v = x$, $u = s$ and $w = t$ ensures that $s$ and $t$ are in distinct components of the forest induced by $\varphi_2^{-1}(2)$.  Thus $\varphi^{-1}(2)$ induces a forest, which proves (G1) and, by the same argument, also (G3).
\end{proof}

We are now ready to prove Theorem \ref{thm:decomp}. 

\begin{proof}[Proof of Theorem \ref{thm:decomp}]
Suppose for a contradiction that this is not the case. Then there exists a counterexample, and also a minimal one, say $S = (G, P, Q, \delta)$. Let $K$ be the cycle bounding the outer face of $G$. By Lemma \ref{lem:2connected}, if $|P| \leq 2$ then $K$ is induced and if $|P| = 3$ then every chord of $K$ is incident with the middle vertex of $P$. 

Let us first consider the case $|V(K)| = 4$.  By Lemma \ref{lem:4cycle}, $G = K = C_4$ and it is easy to extend $\delta$ to a coloring $\varphi$ of $G$ that is valid for $S$.

  So we can assume that $|V(K)| \geq 5$. Since $Q$ is an independent set and $|P| \leq 3$, we can let $r, s, x, z, t$ be vertices consecutive on $K$ with $x \in V(K) \setminus (P \cup Q)$. 
  We distinguish a number of special cases first. 

\bigskip

\bigskip

\parbox{15cm}{\textbf{Case 1:} \emph{$s, z \in Q$ or  $s \in Q$ and $z \in P$  with $\delta(z) = 2$ (or, symmetrically, $z \in Q$ and $s \in P$ with $\delta(s) = 2$)}. In this case, let $G' = G - x$, let
  $X$ be the set of internal vertices of $G$ that are neighbors of $x$, and let  $Q' = X \cup Q$. By Lemma \ref{lem:extended}, $Q'$ is an independent set of $G'$. Thus, the scene $S' = (G', P, Q', \delta)$
  has, by minimality of $S$, a valid coloring $\psi$. Define the coloring $\varphi$ of $S$ by $\varphi(x) = 1$ and $\varphi \restriction G' = \psi$. 
  We only added a single isolated vertex to the induced forest of color 1, thus $\varphi$ is valid for $S$, which is a contradiction. }

\bigskip

\bigskip

\parbox{15cm}{\textbf{Case 2:} \emph{$s \in Q$, $z \not\in P \cup Q$ and $t \not\in Q$}. In this case, let $G' = G - x$ and $X = N_G(x) - s$. Let $Q' = Q \cup X$. By Lemma \ref{lem:extended}, $Q'$ is an
  independent set of $G'$. Thus, the scene $S' = (G', P, Q', \delta)$ has, by minimality of $S$, a valid coloring $\psi$. Define the coloring $\varphi$ of $S$ by $\varphi(x) = 1$ and $\varphi \restriction
  G' = \psi$. Again by the same argument, $\varphi$ is valid for $S$, which is a contradiction.}

\bigskip

\bigskip

\parbox{15cm}{\textbf{Case 3:} \emph{$s, z \not\in P \cup Q$ and $r, t \not\in Q$}. In this case, we apply the argument from the preceding case with $G' = G - x$ and $X = N_G(x)$. }

\bigskip

\bigskip

\parbox{15cm}{\textbf{Case 4:} \emph{$ s\in Q$, $z \not\in P \cup Q$ and $t \in Q$}. In this case, let $G^* = G - \{x, z\}$, and let $Z$ be the set of internal vertices of $G$ that are neighbors of  $x$ or
  $z$. The graph induced by $Z \cup \{s, t\}$ has, by Lemma \ref{lem:4cycle}, at most one edge~$ab$ and, by Lemma \ref{lem:extended}, $\{a, b\} \subseteq Z$. Let $Q^* = Q \cup (Z - \{a\})$, and note by
  Lemma \ref{lem:extended} that $Q^*$ is an independent set. Thus, $S^* = (G^*, P, Q^*, \delta)$ has, by the minimality of $S$, a valid coloring $\psi$. Define the coloring $\varphi$ of $S$ by 
  $\varphi(x) = \varphi(z) = 1$ and $\varphi \restriction G^* = \psi$. We have added an isolated edge to the forest of color~1. So, again, $\varphi$ is valid for $S$, which is a contradiction.} 

\bigskip

\bigskip

\noindent
In the remainder of the proof, we will assume that none of Cases 1--4 apply. Then there are no two consecutive vertices $V(K) - (P \cup Q)$; for suppose that there are two consecutive vertices $u, v \in
  V(K) - (P \cup Q)$, and  let $a$ and $b$ be the other neighbors of $u$ and $v$, respectively, on $K$. If $\{a, b\} \subseteq Q$, then  Case 4 applies. If $a \not\in Q$ and $b \in Q$, then Case~2 applies
  (with $s = b$, $x = v$, $z = u$ and $t = a$). Therefore, by symmetry, we can assume that $\{a, b\} \cap Q = \emptyset$ and thus Case 3 applies. 

So we can assume from now on that there are no two consecutive vertices in $V(K) - (P \cup Q)$. Then Case~1 applies unless $s \in Q$ and $z \in P$ with $\delta(z) = 1$. 
  In such a case, we choose a different $x$, namely the other neighbor of~$S$. This choice works, except in the following two cases. 
\begin{itemize}
\item[(i)] $|V(K)| = 5$, $P = \{r, s\}$ with $\delta(r) = \delta(s) = 1$, $z \in Q$ and $t \not\in P \cup Q$, or
\item[(ii)] $|V(K)|= 5$, $|P| = 3$ with $\delta(s) = 1 \not=\delta(r) = \delta(t) = 2$ and $z \in Q$. 
\end{itemize}

Let us first consider (i). Let $R$ be the set of internal neighbors of $r$, and let $Q' = R \cup \{t\}$. Since $G$ is triangle-free, $Q'$ is an independent set.  Let $\delta'$ be the coloring of $\{s, x, z\}$ defined by $\delta'(s) = 1 \not= \delta'(x) = \delta'(z) = 2$. Then $(G - \{r\}, \{s, x, z\}, Q', \delta')$ is a good scene and, by the minimality of $S$, has a valid coloring $\psi$. The coloring $\varphi$ of $G$ defined by $\varphi \restriction (G - r) = \psi$ and $\varphi(r) = \delta(r) = 1$ is valid for $S$, a contradiction. 

It remains to consider (ii). Let $T$ be the set of internal neighbors of $s$, and let $Q' = T \cup \{x, r\}$. Since $G$ is triangle-free, $Q'$ is an independent set.  Let $\delta'$ be a coloring of $\{z, t\}$ defined by $\delta'(z) = \delta'(t) = 2$. Then the scene $(G - s, \{z, t\}, Q', \delta')$ has, by the minimality of $S$, a valid coloring $\psi$. The coloring $\varphi$ of $G$ defined by $\varphi \restriction (G - s) = \psi$ and $\varphi(s) = \delta(s) =  1$ is valid for $S$, which is our final contradiction. The theorem is proved. 
\end{proof}

\section*{Acknowledgments}

This work was supported by grant 22-17398S of the Czech Science Foundation and by Agence
Nationale de la Recherche (France) under research grant ANR DIGRAPHS ANR-19-CE48-0013-01.

\bibliography{bibliography}{}
\bibliographystyle{abbrv}
 
\end{document}